\documentclass[12pt,a4paper]{article}

\usepackage[utf8]{inputenc}
\usepackage{amssymb, amsthm, amsmath}
\usepackage[english]{babel}
\usepackage[T1]{fontenc}
\usepackage{graphicx}
\usepackage{mathtools}
\usepackage[font=small,labelfont={normal, bf}]{caption}
\usepackage{subcaption}
\usepackage{xcolor}
\usepackage[cm]{fullpage}
\usepackage{autonum}
\usepackage{comment}
\makeatletter
\newcommand{\@giventhatstar}[2]{#1\;\middle|\;#2}
\newcommand{\@giventhatnostar}[3][]{#1#2\;#1|\;#3#1}
\newcommand{\giventhat}{\@ifstar\@giventhatstar\@giventhatnostar}

\makeatother
\newtheorem{theorem}{Theorem}
\newtheorem{proposition}{Proposition}
\theoremstyle{definition}

\newtheorem{example}{Example}
\newtheorem{remark}{Remark}
\newtheorem{lemma}{Lemma}

\title{Multidimensional fractional material derivatives: measure-valued solutions and probabilistic properties}
\author{Hubert Woszczek\thanks{Faculty of Pure and Applied Mathematics, Wroclaw University of Science and Technology, Wyb. Wyspia\'nskiego 27, 50-370 Wroc{\l}aw, Poland\\\underline{Corresponding author:} \texttt{hubert.woszczek@pwr.edu.pl}}}
\date{\today}

\begin{document}
\maketitle

\begin{abstract}
We analyze a nonlocal operator in space and time, called the multidimensional fractional material derivative. We derive its pointwise representation, which allows us to study its other properties. We define an inverse operator, called fractional material integral, and derive its pointwise representation. Furthermore, we analyze a class of linear partial differential equations, which corresponds to deterministic descriptions of the scaling limits of multidimensional L\'evy walks, in which transport is driven by a multidimensional fractional material derivative with a speed vector integrated with respect to a suitable probability measure and a distributional source term. Using Fourier-Laplace transform techniques and a direct convolution-kernel construction, we prove the existence and uniqueness of exponentially integrable measure solutions for measure data. Moreover, we identify a necessary and sufficient condition on the source term for conservation of unit mass and provide separate sufficient conditions for non-negativity and weak convergence to $\delta_0$. We also prove weak continuity and self-similarity of the kernel and continuous dependence on the directional distribution.
\end{abstract}


\section{Introduction}
Nonlocal operators have attracted a lot of attention from the mathematical community, in both pure and applied mathematics. Among them, fractional-order derivatives, which have several different definitions, are among the most popular. In this paper, we focus on the multidimensional fractional derivative, which generalizes the concept of the material derivative, well known in the context of fluid dynamics \cite{batchelor2000introduction}.

The concept of fractional material derivative arises from the study of continuous-time random walk (CTRW) introduced by Montroll and Weiss in \cite{montroll1965random}. Formally, a CTRW process is defined as $X(t) = \sum_{i=1}^{N(t)} J_i$, where $(T_i, J_i)_{i \ge 1}$ are independent and identically distributed (IID) vectors of waiting times ($T_i$) and jump vectors ($J_i$). The number of jumps up to time $t$ is governed by the renewal process $N(t) = \max\{n : T_1 + T_2 + \dots + T_n \le t\}$. As a result, the value of $X(t)$ describes the position of a walker that starts at the origin and then undergoes instantaneous jumps $(J_i)_{i \ge 1}$ that are separated by waiting times $(T_i)_{i \ge 1}$ during which the walker does not change position. In a standard uncoupled CTRW, jumps and waiting times are independent. However, to transition this model into L\'evy walk (LW) processes, we must abandon this independence and enforce a strict coupling between jump lengths and waiting times and assume that waiting times have an infinite mean, scaling as $\mathbb{P}(T_i > t) \sim t^{-\alpha}$ for $0 < \alpha < 1$. The jumps are defined as $J_i = V_i T_i$, where $V_i$ is a random vector distributed on the $d-1$-dimensional unit hypersphere $\mathbb{S}^{d-1}$ according to the measure $\Lambda$. The directions are independent of the waiting times. $V_i$ represents the direction of the $i$-th jump. We can distinguish the three LW processes \cite{MagdziarzTeuerle2015}
\begin{align}
L_{WF}(t) = \sum_{i=1}^{N(t)} J_i,\quad L_{JF}(t) = \sum_{i=1}^{N(t)+1} J_i,\quad L(t) = \sum_{i=1}^{N(t)} J_i + V_{N(t)+1}\left(t-\sum_{i=1}^{N(t)}T_i\right)
\end{align}
called wait-first LW, jump-first LW and continuous (or standard) LW, respectively. Note that while $L(t)$ has continuous trajectories, both $L_{WF}(t)$ and $L_{JF}(t)$ feature discontinuous c\`adl\`ag trajectories (right-continuous with left limits). The joint partial-sum convergence below is considered in the Skorokhod space $\mathcal D([0,\infty),\mathbb R^{d+1})$ endowed with the topology $J_1$ \cite{billingsley1999convergence}. As $n\to\infty$, the partial sums of the jumps and waiting times jointly converge to a $(d+1)$-dimensional stochastic process \cite{MagdziarzTeuerle2015}
\[
\left(a_n^{-1}\sum_{i=1}^{[nt]}J_i,\ a_n^{-1}\sum_{i=1}^{[nt]}T_i\right)
\xrightarrow{J_1}(L_\alpha(t),S_\alpha(t)),
\]
where the normalization is chosen so that
$n\mathbb P(T_1>a_n)\to1/\Gamma(1-\alpha)$; under the stated tail assumption one may take $a_n=(n\Gamma(1-\alpha))^{1/\alpha}$.
Here $L_\alpha(t)$ is an $\alpha$-stable process with the Fourier transform $\mathbb{E}[e^{i\xi\cdot L_\alpha(t)}] = \exp\{-t\int_{\mathbb{S}^{d-1}}|\xi\cdot b|^\alpha C^* (1-i\operatorname{sgn}(\xi\cdot b)\tan(\pi\alpha/2))\Lambda(db)\}$, where $C^*=\cos(\pi\alpha/2)$. $S_\alpha(t)$ is an $\alpha$-stable subordinator (a strictly increasing L\'evy process) with the Laplace transform $\mathbb{E}[e^{-sS_\alpha(t)}] = e^{-ts^\alpha}$. Since the original jumps and waiting times were identical in magnitude ($|J_i| = T_i$), $L_\alpha(t)$ and $S_\alpha(t)$ are strongly dependent. The scaling limits of the three LW variants are defined using subordination (evaluating the jump process $L_\alpha$ at an operational time dictated by the inverse subordinator $S_\alpha^{-1}(t) = \inf\{\tau : S_\alpha(\tau) > t\}$). For example, the jump-first LW converges \cite{MagdziarzTeuerle2015}
\[
a_n^{-1}L_{JF}(a_nt)\xrightarrow{J_1}L_\alpha(S_\alpha^{-1}(t)).
\]

The link from LW to the fractional material derivative is established through the L\'evy-Khintchine exponent $\psi(\xi,s)$ of the joint limiting process $(L_\alpha,S_\alpha)$. The joint Fourier-Laplace transform of the law of $(L_\alpha(u),S_\alpha(u))$ is $e^{-u\psi(\xi,s)}$, where the exponent is given by
\[
\psi(\xi,s)=\int_{\mathbb S^{d-1}}(s-ib\cdot\xi)^\alpha\Lambda(db).
\]
This specific algebraic symbol $\psi(\xi,s)$ acts as the Fourier-Laplace representation of a pseudo-differential operator. Specifically, we define the fractional material derivatives via their transforms
\begin{equation}\label{def:fracMatDer}
\mathcal{FL}\left\{\left(\frac\partial{\partial t}+b\cdot\nabla\right)^\alpha u(x,t)\right\}(\xi,s)=(s-ib\cdot\xi)^\alpha\hat U(\xi,s).
\end{equation}
In \eqref{def:fracMatDer}, $u$ is extended by zero to negative times and the operator is a causal distributional operator. Its restriction to $t>0$ is the fractional material derivative. A nonzero fractional initial trace contributes a distribution supported at $t=0$ to the causal equation. All complex powers are undesrtood as the principal branch.
This leads us to the definition of the fractional material derivative operator. Moreover, the probability laws of the scaling limits of all three LW variants (with density $u(x,t)$ when it exists) are measure solutions to the pseudo-differential equation
\[
\int_{\mathbb S^{d-1}}\left(\frac\partial{\partial t}+b\cdot\nabla\right)^\alpha u(x,t)\Lambda(db)=f(x,t)
\]
with the initial condition $u(\cdot,t)\rightharpoonup\delta_0$ as $t\to0^+$. The source (measure) $f(x,t)$ accounts for the differences between the wait-first, jump-first, and continuous schemes. Equations with the fractional material derivative in one spatial dimension were studied in \cite{Plociniczak2024, Plociniczak2026Convex}.

The anomalous diffusion processes can be characterized by the mean squared displacement (MSD) of a particle starting at the origin, which deviates from normal diffusion with linear MSD, namely $\langle x^{2}(t)\rangle\sim t$, and exhibits a power law of the form
\begin{equation}
\langle x^{2}(t)\rangle\sim t^\beta,\quad\beta\ne1,
\end{equation}
see \cite{MetzlerKlafter}. When $0<\beta<1$, we call this process \emph{subdiffusion} and for $1<\beta<2$, \emph{superdiffusion}. Subdiffusion is a natural tool for describing diffusion in heterogeneous or disordered environments and has been found in many fields such as porous media \cite{plociniczak2014approximation, plociniczak2015analytical, plociniczak2019derivation, pachepsky2000simulating, El20}, neurological tissues \cite{magin2010fractional}, polymers \cite{muller2011nonlinear}, single particle tracking in biophysics \cite{tabei2013intracellular, Sun17, wong2004anomalous}, plasma physics \cite{Del05}, astrophysics \cite{lawrence1993anomalous}, chemotaxis \cite{langlands2010fractional} and financial mathematics \cite{jacquier2020anomalous}. On the other hand, Superdiffusion appears in processes, where the spreading of the particles is faster than in normal diffusion, often due to long-range dependence or heavy-tailed jump distributions. Natural models for superdiffusive processes are L\'evy flights and L\'evy walks, and they appear in many applications such as turbulent transport in fluids and plasmas \cite{shlesinger1993strange, zaburdaev2015levy}, animal foraging patterns \cite{viswanathan1999optimizing}, human travel and mobility \cite{brockmann2006scaling, gonzalez2008understanding}, ecology and biology \cite{humphries2014optimal, reynolds2009levy, alves2016transient}, light transport in random media \cite{mercadier2009levy, barthelemy2008levy}, and financial time series \cite{mantegna1995scaling}.

Differential equations with nonlocal operators also attract increasing interest from a mathematical point of view. They have a rich structure and present additional difficulties in the considered equations. It is not our aim to give a complete overview of the literature on fractional differential equations, but rather to list some important papers from theoretical and numerical studies of such equations. The solvability of the linear anomalous diffusion equation was studied in \cite{Sak11}. Another interesting research direction concerns time-fractional nonlinear diffusion equations, for which decay properties have been studied and different types of solutions have been introduced \cite{allen2016parabolic, wittbold2021bounded, vergara2015optimal, akagi2019fractional, dipierro2019decay,plociniczak2018existence}. Along with interesting analytical challenges that require a careful setting in appropriate functional spaces, the numerical study of nonlocal equations also produces interesting research problems. The numerical difficulties in such studies come from the lack of sufficient smoothness of the solution and the more stringent stability conditions. However, in the last two decades ideas have appeared on how to overcome these difficulties. The interested reader can consult the following papers (and references therein) for finite element approaches \cite{Jin19a, Mus18, plociniczak2022error,plociniczak2024fully, plociniczak2023linear}, spectral schemes \cite{Lin07}, and fast convolution quadrature methods \cite{lopez2025convolution, cuesta2006convolution}.

The present paper considers arbitrary probability measures $\Lambda$ on the sphere and finite signed spatial measure data. In contrast to the one-dimensional setting of \cite{Plociniczak2024,Plociniczak2026Convex}, the directional distribution may be continuous, discrete, or supported on a proper subspace. We construct a weakly continuous positive kernel, verify the averaged fractional initial trace without assuming separate directional traces, and prove continuous dependence on $\Lambda$. The latter result justifies approximation of an angular distribution by finitely many velocities. We also separate the condition for unit mass from sufficient conditions for positivity and concentration at the origin.

Let $L_{loc}^1(\Omega)$ be the space of locally integrable functions on $\Omega$, that is, functions integrable on every compact set $K\subseteq\Omega$. Define the \emph{fractional integral} of order $\alpha>0$ of $y\in L_{loc}^1(\mathbb R_+)$ given by
\begin{equation}
I^\alpha y(t)=\frac1{\Gamma(\alpha)}\int_0^t(t-s)^{\alpha-1}y(s)\,ds.
\end{equation}
Throughout the paper, we use the Fourier transform convention
\begin{equation}
\widehat h(\xi)=\mathcal F\{h\}(\xi):=\int_{\mathbb R^d}e^{i\xi\cdot x}h(x)\,dx.
\end{equation}
The fractional \emph{Riemann-Liouville derivatives} are defined by
\begin{equation}
D^\alpha y(t)=\frac{d^n}{dt^n}I^{n-\alpha}y(t),\quad n=\lceil\alpha\rceil\normalcolor,
\end{equation}
where $\lceil\alpha\rceil$ is the smallest integer greater than or equal to $\alpha$. In particular, for the physically interesting case with $0<\alpha<1$ we have
\begin{equation}
D^\alpha y(t)=\frac1{\Gamma(1-\alpha)}\frac d{dt}\int_0^t(t-s)^{-\alpha}y(s)\,ds.
\end{equation}
Finally, we introduce the function spaces that we will work on below. Since the fractional material derivative is normally defined via the Fourier-Laplace multiplier as in \eqref{def:fracMatDer}, it is necessary to define the spaces in which it is well-defined. The space in which the Fourier transform naturally lives is the \emph{Schwartz space} of functions that decay faster than algebraically, more specifically, we define
\begin{equation}
\mathcal S(\mathbb R^d)=\left\{y\in C^\infty(\mathbb R^d):\text{for all multiindices }\beta,\gamma\in\mathbb N_0^d,\ \sup_{x\in\mathbb R^d}|x^\beta\partial^\gamma y(x)|<\infty\right\},
\end{equation}
and by $\mathcal S'$ its dual. The Lebesgue space of functions integrable to the power $1\leq p\leq\infty$ is denoted by $L^p(\Omega)$ with $\Omega\subseteq\mathbb R^d$ with $d\geq1$. The space of continuous functions with the continuous $k$-th derivative is $C^{(k)}(\Omega)$ and, similarly, the space of continuous functions that vanish at infinity is denoted by $C_0(\mathbb R^d)$. Also, $\mathcal M(\mathbb R^d)$ denotes the space of finite signed Borel measures on $\mathbb R^d$, endowed with the total variation norm
\begin{equation}
\|\mu\|_{\mathcal M(\mathbb R^d)}:=|\mu|(\mathbb R^d).
\end{equation}

\section{Multidimensional fractional material derivative and its properties}
Let $x\in\mathbb R^d$, $\mathbb S^{d-1}=\{x\in\mathbb R^d:\|x\|=1\}$ be the unit sphere, $b\in\mathbb S^{d-1}$. The Fourier-Laplace transform of the multidimensional fractional material derivative $\left(\frac\partial{\partial t}+b\cdot\nabla\right)^\alpha u(x,t)$ in the direction $b$ is given by
\begin{equation}\label{fltfmd}
\mathcal F\mathcal L\left\{\left(\frac\partial{\partial t}+b\cdot\nabla\right)^\alpha u(x,t)\right\}(\xi,s)=(s-ib\cdot\xi)^\alpha\mathcal F\mathcal L\{u(x,t)\}(\xi,s).
\end{equation}
Here and below the multiplier idetity refers to the causal distribution and pointwise formulas refer to its restriction to $t>0$. In the next theorem, we will prove the pointwise representation of such an operator.
\begin{theorem}[Pointwise representation]
Let $u(\cdot,t)\in\mathcal S(\mathbb R^d)$ for all $t\in\mathbb R_+$ and satisfy the uniform exponential bound
\begin{equation}\label{uniformexpbound_pointwise}
\|u(\cdot,t)\|_\infty\leq Me^{at},\quad t\geq0,
\end{equation}
for some constants $M>0$ and $a\in\mathbb R$. Assume also that $u\in C^1(\mathbb R^d\times[0,\infty))$. Then, for $0<\alpha<1$ we have for $t>0$
\begin{equation}\label{pointwiserepeq}
\left(\frac\partial{\partial t}+b\cdot\nabla\right)^\alpha u(x,t)=\frac1{\Gamma(1-\alpha)}\left(\frac\partial{\partial t}+b\cdot\nabla\right)\int_0^t(t-s)^{-\alpha}u(x-b(t-s),s)\,ds.
\end{equation}
\begin{proof}
We start from the following factorization of Fourier-Laplace transform of fractional material derivative:
\begin{equation}
(s-ib\cdot\xi)^\alpha\mathcal F\mathcal L\{u\}=(s-ib\cdot\xi)(s-ib\cdot\xi)^{\alpha-1}\mathcal F\mathcal L\{u\}.
\end{equation}
It is a well-known fact that
\begin{equation}
\mathcal L^{-1}\{(s-a)^{\alpha-1}\}(t)=e^{at}\frac{t^{-\alpha}}{\Gamma(1-\alpha)},\quad\Re s>\Re a.
\end{equation}
From this and the fact that inverse Laplace transform of multiplication is a convolution of inverse Laplace transforms, we have
\begin{equation}
\hat I(\xi,t):=\mathcal L^{-1}\{(s-ib\cdot\xi)^{\alpha-1}\mathcal F\mathcal L\{u\}\}(t)=\frac1{\Gamma(1-\alpha)}\int_0^t(t-\tau)^{-\alpha}e^{i(t-\tau)b\cdot\xi}\hat u(\xi,\tau)\,d\tau.
\end{equation}
The transforms and convolution identity may be interpreted in $\mathcal S'$. Let us recall that for a function $q$ that admits a Fourier transform, we have
\begin{equation}
\mathcal F\{q(x-rb)\}(\xi)=e^{irb\cdot\xi}\hat q(\xi).
\end{equation}
Thus, we have
\begin{equation}
e^{i(t-\tau)b\cdot\xi}\hat u(\xi,\tau)=\mathcal F\{u(x-b(t-\tau),\tau)\}(\xi).
\end{equation}
Hence,
\begin{equation}
\hat I(\xi,t)=\mathcal F\left\{\frac1{\Gamma(1-\alpha)}\int_0^t(t-\tau)^{-\alpha}u(x-b(t-\tau),\tau)\,d\tau\right\}.
\end{equation}
From the above considerations, we have
\begin{equation}
\mathcal F\mathcal L\{I\}(\xi,s)=(s-ib\cdot\xi)^{\alpha-1}\mathcal F\mathcal L\{u\}(\xi,s).
\end{equation}
Multiplying the above equation by $(s-ib\cdot\xi)$ and bearing in mind that it is the multiplier of the operator $\left(\frac\partial{\partial t}+b\cdot\nabla\right)$, with $I(x,0)=0$, we obtain
\begin{equation}
(s-ib\cdot\xi)^\alpha\mathcal F\mathcal L\{u\}(\xi,s)=\mathcal F\mathcal L\left\{\left(\frac\partial{\partial t}+b\cdot\nabla\right)I\right\}(\xi,s).
\end{equation}
To justify using the classical derivative, fix $C\in\mathbb R^d$ and set $V_C(t)=u(C+bt,t)$. Then
\[
I(C+bt,t)=I^{1-\alpha}V_C(t),\quad
\frac d{dt}I^{1-\alpha}V_C(t)
=\frac{t^{-\alpha}}{\Gamma(1-\alpha)}V_C(0)+I^{1-\alpha}V_C'(t).
\]
The right-hand side is well-defined for $t>0$ by the assumed $C^1$ regularity. This also identifies the distributional derivative with the pointwise formula. From the definition of fractional material derivative via Fourier-Laplace transform, we have
\begin{equation}
\mathcal F\mathcal L\left\{\left(\frac\partial{\partial t}+b\cdot\nabla\right)^\alpha u\right\}=\mathcal F\mathcal L\left\{\left(\frac\partial{\partial t}+b\cdot\nabla\right)I\right\},
\end{equation}
which completes the proof.
\end{proof}
\end{theorem}
\begin{remark}
Notice that formula \eqref{pointwiserepeq} does not require decay at spatial infinity. More generally, the integral is well-defined whenever, for each $(x,t)$, the map $s\mapsto(t-s)^{-\alpha}u(x-b(t-s),s)$ belongs to $L^1((0,t))$. If the resulting integral is not classically differentiable, the operator in \eqref{pointwiserepeq} can be understood in the distributional sense, provided the integral is locally integrable in space and time.
\end{remark}
Next, we will prove some basic properties of the fractional material derivative. We start from continuity with respect to $\alpha$ and action on the traveling wave. We assume that $0<\alpha<1$.
\begin{proposition}
\begin{enumerate}
\item Let $u\in C^1(\mathbb R^d\times[0,\infty))$ and assume that the expressions below are well-defined. Then for $t>0$
\begin{equation}
\lim_{\alpha\to0^+}\left(\frac\partial{\partial t}+b\cdot\nabla\right)^\alpha u(x,t)=u(x,t),\quad
\lim_{\alpha\to1^-}\left(\frac\partial{\partial t}+b\cdot\nabla\right)^\alpha u(x,t)=\left(\frac\partial{\partial t}+b\cdot\nabla\right)u(x,t).
\end{equation}
\item Let $u(x,t)=U(x-bt)$, where $U\in C^1(\mathbb R^d)$, then
\begin{equation}
\left(\frac\partial{\partial t}+b\cdot\nabla\right)^\alpha u(x,t)=\frac{t^{-\alpha}}{\Gamma(1-\alpha)}U(x-bt).
\end{equation}
\item Define $u_\lambda(x,t):=u(\lambda x,\lambda t)$, $\lambda>0$. Then,
\begin{equation}
\left(\left(\frac\partial{\partial t}+b\cdot\nabla_x\right)^\alpha u_\lambda\right)(x,t)=\lambda^\alpha\left(\left(\frac\partial{\partial t}+b\cdot\nabla_x\right)^\alpha u\right)(\lambda x,\lambda t).
\end{equation}
\end{enumerate}
\begin{proof}
\begin{enumerate}
\item It follows from the properties of the Riemann-Liouville fractional derivative applied to $V_C(t)=u(C+bt,t)$.
\item It is sufficient to observe
\begin{equation}
u(x-b(t-s),s)=U(x-b(t-s)-bs)=U(x-bt).
\end{equation}
\item Using pointwise representation \eqref{pointwiserepeq}, from direct calculations we have
\begin{equation}
\begin{aligned}
\left(\left(\frac\partial{\partial t}+b\cdot\nabla_x\right)^\alpha u_\lambda\right)(x,t)
&=\frac1{\Gamma(1-\alpha)}\left(\frac\partial{\partial t}+b\cdot\nabla_x\right)\int_0^t(t-s)^{-\alpha}u_\lambda(x-b(t-s),s)\,ds\\
&=\frac1{\Gamma(1-\alpha)}\left(\frac\partial{\partial t}+b\cdot\nabla_x\right)\int_0^t(t-s)^{-\alpha}u(\lambda(x-b(t-s)),\lambda s)\,ds.
\end{aligned}
\end{equation}
Now, by the change of variable $v=\lambda s$, we have
\begin{equation}
\begin{aligned}
\left(\left(\frac\partial{\partial t}+b\cdot\nabla_x\right)^\alpha u_\lambda\right)(x,t)
&=\frac{\lambda^{\alpha-1}}{\Gamma(1-\alpha)}\left(\frac\partial{\partial t}+b\cdot\nabla_x\right)\int_0^{\lambda t}(\lambda t-v)^{-\alpha}u(\lambda x-b(\lambda t-v),v)\,dv\\
&=\lambda^\alpha\left(\left(\frac\partial{\partial t}+b\cdot\nabla_x\right)^\alpha u\right)(\lambda x,\lambda t),
\end{aligned}
\end{equation}
which completes the proof.
\end{enumerate}
\end{proof}
\end{proposition}

\begin{proposition}
Let $\phi\in C(\mathbb R^d)$ and $f\in C(\mathbb R^d\times[0,\infty))$.
Then, the solution class consists of locally integrable functions $u$ such that, with $V(C,t)=u(C+bt,t)$, the map
$C\mapsto V(C,\cdot)$ is continuous into $L^1(0,T)$ for every $T>0$, and $I^{1-\alpha}V$ has a continuous representative on $\mathbb R^d\times[0,\infty)$ with $I^{1-\alpha}V(C,0)=\phi(C)$.
Then, the unique solution up to equality almost everywhere of the problem
\begin{equation}\label{simplepde}
\begin{cases}
\left(\dfrac\partial{\partial t}+b\cdot\nabla\right)^\alpha u(x,t)=f(x,t),&x\in\mathbb R^d,\ t>0,\\
\displaystyle\lim_{t\to0^+}I_b^{1-\alpha}u(x,t)=\phi(x),
\end{cases}
\end{equation}
with the equation understood distributionally on $t>0$, is given by
\begin{equation}\label{exactsol}
u(x,t)=\frac{t^{\alpha-1}}{\Gamma(\alpha)}\phi(x-bt)+\frac1{\Gamma(\alpha)}\int_0^t(t-s)^{\alpha-1}f(x-b(t-s),s)\,ds.
\end{equation}
\begin{proof}
We will use the classical method of characteristics. Fix a constant $C\in\mathbb R^d$ and consider the characteristic $x=bt+C$. Set $V(t):=u(bt+C,t)$. Then
\begin{equation}
I^{1-\alpha}V(t)=\frac1{\Gamma(1-\alpha)}\int_0^t(t-s)^{-\alpha}u(C+bs,s)\,ds=I_b^{1-\alpha}u(bt+C,t).
\end{equation}
Therefore, by \eqref{pointwiserepeq}, problem \eqref{simplepde} reduces along each characteristic to
\begin{equation}
D_t^\alpha V(t)=f(bt+C,t),\quad\lim_{t\to0^+}I_t^{1-\alpha}V(t)=\phi(C).
\end{equation}
The continuity assumptions allow the distributional equation to be restricted to every characteristic and the continuous extension of the fractional integral identifies its trace along $x=C+bt$.
Applying the fractional integral $I_t^\alpha$ and using the standard Riemann-Liouville inversion formula (see \cite{KilbasSrivastavaTrujillo2006}, formula (2.1.40)), we obtain
\begin{equation}
V(t)=\frac{t^{\alpha-1}}{\Gamma(\alpha)}\phi(C)+\frac1{\Gamma(\alpha)}\int_0^t(t-s)^{\alpha-1}f(bs+C,s)\,ds.
\end{equation}
Finally, substituting $C=x-bt$ gives \eqref{exactsol}.
\end{proof}
\end{proposition}
Next, we will prove the conservation law.
\begin{proposition}
Let $\Omega\subset\mathbb R^d$ be a bounded Lipschitz domain and $u(x,t)$ be the solution of \eqref{simplepde}. Assume that the material integral $I=I_b^{1-\alpha}u$ belongs to $C^1(\overline\Omega\times(0,T])$ and that $f(\cdot,t)$ is integrable on $\Omega$. Then, we have
\begin{equation}
\frac d{dt}\int_\Omega I(x,t)\,dx+\int_{\partial\Omega}(n\cdot b)I(x,t)\,dS=\int_\Omega f(x,t)\,dx,
\end{equation}
where $\partial\Omega$ is the boundary of $\Omega$ and $n$ is the outward normal vector to $\partial\Omega$ and $I(x,t)=\frac1{\Gamma(1-\alpha)}\int_0^t(t-s)^{-\alpha}u(x-b(t-s),s)\,ds$.
\begin{proof}
Using pointwise representation \eqref{pointwiserepeq} to the PDE \eqref{simplepde}, we have
\begin{equation}
\frac\partial{\partial t}I(x,t)+b\cdot\nabla I(x,t)=f(x,t).
\end{equation}
Integrating over $\Omega$, we have
\begin{equation}
\frac d{dt}\int_\Omega I(x,t)\,dx+\int_\Omega b\cdot\nabla I(x,t)\,dx=\int_\Omega f(x,t)\,dx.
\end{equation}
Applying the divergence theorem to the equation above, we obtain the thesis.
\end{proof}
\end{proposition}
Now, we will derive a pointwise representation the inverse operator it in the same spirit as for the fractional material derivative. We will call such an operator the \emph{fractional material integral}.
\begin{proposition}
Let $u(\cdot,t)\in\mathcal S(\mathbb R^d)$ for all $t\in\mathbb R_+$ and satisfy the uniform exponential bound
\begin{equation}\label{uniformexpbound_integral}
\|u(\cdot,t)\|_\infty\leq Me^{at},\quad t\geq0,
\end{equation}
for some constants $M>0$ and $a\in\mathbb R$. Assume that $u$ is jointly Borel measurable and define $I_b^\alpha$ by the causal multiplier $(s-ib\cdot\xi)^{-\alpha}$. Then, for $0<\alpha<1$ we have
\begin{equation}
I_b^\alpha u(x,t)=\frac1{\Gamma(\alpha)}\int_0^t(t-s)^{\alpha-1}u(x-b(t-s),s)\,ds.
\end{equation}
\begin{proof}
It is a well-known fact that
\begin{equation}
\mathcal L^{-1}\{(s-a)^{-\alpha}\}(t)=e^{at}\frac{t^{\alpha-1}}{\Gamma(\alpha)},\quad\Re s>\Re a.
\end{equation}
From this and the fact that inverse Laplace transform of multiplication is a convolution of inverse Laplace transforms, we have
\begin{equation}
\hat I^\alpha(\xi,t):=\mathcal L^{-1}\{(s-ib\cdot\xi)^{-\alpha}\mathcal F\mathcal L\{u\}\}(t)=\frac1{\Gamma(\alpha)}\int_0^t(t-\tau)^{\alpha-1}e^{i(t-\tau)b\cdot\xi}\hat u(\xi,\tau)\,d\tau.
\end{equation}
As above, this identity is valid in the space of tempered distributions. Let us recall that for a function $q$ that admits a Fourier transform, we have
\begin{equation}
\mathcal F\{q(x-rb)\}(\xi)=e^{irb\cdot\xi}\hat q(\xi).
\end{equation}
Thus, we have
\begin{equation}
e^{i(t-\tau)b\cdot\xi}\hat u(\xi,\tau)=\mathcal F\{u(x-b(t-\tau),\tau)\}(\xi).
\end{equation}
Hence,
\begin{equation}
\hat I^\alpha(\xi,t)=\mathcal F\left\{\frac1{\Gamma(\alpha)}\int_0^t(t-\tau)^{\alpha-1}u(x-b(t-\tau),\tau)\,d\tau\right\},
\end{equation}
which completes the proof.
\end{proof}
\end{proposition}
\section{Weighted combination of multidimensional fractional material derivatives}
In this section, we want to consider a more general situation, where a vector of directions $b$ is distributed according to some probability measure $\Lambda$. In this case, the operator is defined as
\begin{equation}
\left(\frac\partial{\partial t}+b\cdot\nabla\right)_\Lambda^\alpha u(x,t):=\int_{\mathbb S^{d-1}}\left(\frac\partial{\partial t}+b\cdot\nabla\right)^\alpha u(x,t)\Lambda(db).
\end{equation}
The goal of this section is to analyze the equation
\begin{equation}\label{combinationeq}
\begin{cases}
\left(\dfrac\partial{\partial t}+b\cdot\nabla\right)_\Lambda^\alpha u(x,t)=f(x,t),&x\in\mathbb R^d,\ t>0,\\
\displaystyle\lim_{t\to0^+}\int_{\mathbb S^{d-1}}I_b^{1-\alpha}u(\cdot,t)\,\Lambda(db)=g,
\end{cases}
\end{equation}
for suitable data $f$ and $g$. Before, we start the analysis of the equation, let us introduce some auxilary notation
For a map $T$ and a measure $\mu$, $T_\#\mu$ denotes the pushforward measure. Let $D_cx=cx$ be a dilation of $x$ by $c$ and
\[
C_\Lambda=\overline{\operatorname{conv}}(\operatorname{supp}\Lambda)\subseteq\overline{B_1(0)}.
\]
We first construct a version of the kernel defined at every positive time. This kernel will be crucial tool in the proof of existence and uniqueness of solution.
\begin{lemma}[Kernel construction]\label{kernelcanonical}
Let $0<\alpha<1$ and let $\Lambda$ be a probability measure on $\mathbb S^{d-1}$. Then,  there exists a probability measure $\rho_\Lambda$ supported in $C_\Lambda$ such that
\begin{equation}\label{canonicalkernel}
G_t=\frac{t^{\alpha-1}}{\Gamma(\alpha)}(D_t)_\#\rho_\Lambda,\quad t>0.
\end{equation}
Moreover, the family $G_t$ is weak-$*$ continuous on $(0,\infty)$, satisfies
\begin{equation}\label{kernelscaling}
G_{ct}=c^{\alpha-1}(D_c)_\#G_t,\quad c,t>0,
\end{equation}
and has support in $tC_\Lambda$. Its Fourier-Laplace transform is given by
\begin{equation}\label{kerneltransform}
\int_0^\infty e^{-st}\widehat{G_t}(\xi)\,dt
=\frac1{\zeta_\Lambda(\xi,s)},\quad
\zeta_\Lambda(\xi,s)=\int_{\mathbb S^{d-1}}(s-ib\cdot\xi)^\alpha\Lambda(db),\quad\Re s>0.
\end{equation}
\end{lemma}
\begin{proof}
Let us consider the L\'evy measure
\[
\Pi(dt,dx)=\frac{\alpha}{\Gamma(1-\alpha)}
t^{-1-\alpha}\,dt
\int_{\mathbb S^{d-1}}\delta_{tb}(dx)\,\Lambda(db).
\]
Since $|tb|=t$ and $\Lambda$ is a probability measure,
\[
\int \bigl(1\wedge(t+|x|)\bigr)\,\Pi(dt,dx)
=
\frac{\alpha}{\Gamma(1-\alpha)}
\int_0^\infty(1\wedge 2t)t^{-1-\alpha}\,dt
=
\frac{2^\alpha}{(1-\alpha)\Gamma(1-\alpha)}
<\infty.
\]
Thus, choosing zero Gaussian component and zero deterministic
drift in the uncompensated L\'evy--It\^o representation,
we obtain a L\'evy process $(S_r,L_r)$ whose paths have finite
variation on compact intervals (see
\cite[Theorem~21.9]{Sato1999}).
Since all temporal jumps are positive, this process takes
values in $[0,\infty)\times\mathbb R^d$. From the identity 
\[
\frac{\alpha}{\Gamma(1-\alpha)}
\int_0^\infty(1-e^{-zt})t^{-1-\alpha}\,dt
=z^\alpha,\quad \Re z>0,
\]
where the principal branch is used, we obtain 
Fourier-Laplace exponent in form
\[
\begin{aligned}
\zeta_\Lambda(\xi,s)
=\int_{(0,\infty)\times\mathbb R^d}
\bigl(1-e^{-st+i\xi\cdot x}\bigr)\,\Pi(dt,dx)=\int_{\mathbb S^{d-1}}
(s-ib\cdot\xi)^\alpha\,\Lambda(db),
\quad \Re s>0.
\end{aligned}
\]
Let $\mu_r$ be the law of $(S_r,L_r)$. By the uncompensated
L\'evy--It\^o representation,
\[
S_r=\sum_{0<q\leq r}\Delta S_q,
\quad
L_r=\sum_{0<q\leq r}\Delta L_q,
\]
where $\Delta$ operation here means difference between process and its left limit and the sums converge absolutely almost surely. The support
of $\Pi$ implies that
$\Delta L_q=\Delta S_q\,b_q$ for some
$b_q\in\operatorname{supp}\Lambda$. Consequently, on $\{S_r>0\}$, we have
\[
\frac{L_r}{S_r}
=\sum_{0<q\leq r}\frac{\Delta S_q}{S_r}\,b_q,
\quad
\sum_{0<q\leq r}\frac{\Delta S_q}{S_r}=1.
\]
The first countable convex combination above belongs to
$C_\Lambda=\overline{\operatorname{conv}}(\operatorname{supp}\Lambda)$.
Indeed, it is a limit of finite convex combinations, obtained
by assigning the remaining weight to any fixed point of
$\operatorname{supp}\Lambda$.
On $\{S_r=0\}$, all jumps up to time $r$ vanish and hence $L_r=0$.
Thus $L_r\in S_rC_\Lambda$ almost surely, and
\[
\mu_r\bigl(\{(t,x):t\geq0,\ x\in tC_\Lambda\}\bigr)=1.
\]
Setting $\xi=0$ in the Fourier-Laplace transform gives
\[
\mathbb E[e^{-sS_r}]
=e^{-r\zeta_\Lambda(0,s)}
=e^{-rs^\alpha},\quad s>0,
\]
because $\Lambda(\mathbb S^{d-1})=1$.
Together with its nondecreasing paths, this identifies $S$
as the standard $\alpha$-stable subordinator. For $\Re s>0$ every $(s-ib\cdot\xi)^\alpha$ has positive real part, so
\begin{equation}\label{zetapositive}
\Re\zeta_\Lambda(\xi,s)>0.
\end{equation}
Next, we define the potential measure
\[
Q(E)=\int_0^\infty\mu_r(E)\,dr.
\]
Its Fourier-Laplace transform is $1/\zeta_\Lambda(\xi,s)$. In particular, its time marginal is
\[
Q(dt,\mathbb R^d)=\frac{t^{\alpha-1}}{\Gamma(\alpha)}\,dt,
\]
by uniqueness of the Laplace transform of locally finite positive measures. This also proves finiteness of $Q$ on each set $0<t\leq T$. Homogeneity of the exponent implies
$(D_c^{\mathrm{st}})_\#\mu_r=\mu_{c^\alpha r}$, where $D_c^{\mathrm{st}}(t,x)=(ct,cx)$, and consequently $Q(cE)=c^\alpha Q(E)$. For a Borel set $A\subseteq C_\Lambda$, write
\[
F_A(T)=Q\{(t,x):0<t\leq T,\ x/t\in A\}.
\]
Then $F_A(T)=T^\alpha F_A(1)$. Defining
$\rho_\Lambda(A)=\Gamma(\alpha+1)F_A(1)$ gives a probability measure and, by differentiating $F_A$ in $T$, we have
\[
Q\{(t,x):t\in dt,\ x/t\in A\}
=\frac{t^{\alpha-1}}{\Gamma(\alpha)}\,dt\,\rho_\Lambda(A).
\]
Thus $Q(dt,dx)=G_t(dx)\,dt$ with $G_t$ defined by \eqref{canonicalkernel}. Moreover this proves \eqref{kerneltransform}. The support and scaling identities follow immediately. For $\varphi\in C_0(\mathbb R^d)$, pairing
\[
\langle G_t,\varphi\rangle=\frac{t^{\alpha-1}}{\Gamma(\alpha)}\int_{C_\Lambda}\varphi(tz)\,\rho_\Lambda(dz)
\]
is continuous for $t>0$ by dominated convergence theorem. Hence, the stated kernel formula is valid for all $t>0$.
\end{proof}

\begin{theorem}\label{exuniq}
Let $0<\alpha<1$, and let $\Lambda$ be a probability measure on $\mathbb S^{d-1}$. Let
\begin{equation}
g\in\mathcal M(\mathbb R^d),
\end{equation}
and let
\begin{equation}
f:(0,\infty)\to\mathcal M(\mathbb R^d),\quad t\mapsto f(\cdot,t),
\end{equation}
be weak-$*$ measurable. Assume that there exist constants $M>0$, $a\in\mathbb R$ and $0\leq\beta<1$ such that
\begin{equation}\label{uniformexpbound}
\|f(\cdot,t)\|_{\mathcal M(\mathbb R^d)}\leq Mt^{-\beta}e^{at},\quad t>0.
\end{equation}
For $b\in\mathbb S^{d-1}$, define the fractional material integral $I_b^{1-\alpha}$ in the weak sense by
\begin{equation}\label{fractionalmaterialintegral}
\left\langle I_b^{1-\alpha}u(\cdot,t),\varphi\right\rangle
=\frac1{\Gamma(1-\alpha)}\int_0^t(t-\tau)^{-\alpha}\left\langle u(\cdot,\tau),\varphi(\cdot+b(t-\tau))\right\rangle\,d\tau,
\end{equation}
for every $\varphi\in C_0(\mathbb R^d)$, whenever the right-hand side is finite.
Then, the solution class consists of weak-$*$ measurable measure-valued functions satisfying, 
\begin{equation}\label{solutionclass}
\int_0^\infty e^{-\sigma t}\|u(\cdot,t)\|_{\mathcal M}\,dt<\infty,
\end{equation}
for some $\sigma>0$. In particular, $u$ is integrable in total variation on every interval $(0,T)$. For almost every $t>0$, we define locally integrable measure-valued fields
\begin{equation}\label{aggregatefields}
W(t)=\int_{\mathbb S^{d-1}}I_b^{1-\alpha}u(t)\,\Lambda(db),\quad
J(t)=\int_{\mathbb S^{d-1}}b\,I_b^{1-\alpha}u(t)\,\Lambda(db).
\end{equation}
The we can rewrite the equation \eqref{combinationeq} in form
\begin{equation}\label{weakcombinationeq}
\partial_tW+\operatorname{div}_xJ=f\quad\hbox{in }\mathcal D'(\mathbb R^d\times(0,\infty)).
\end{equation}
with the initial condition in the sense that $W$ has a representative with
\begin{equation}\label{weakRLinitialcondition}
\lim_{t\to0^+}\langle W(t),\varphi\rangle=\langle g,\varphi\rangle,\quad\varphi\in C_0(\mathbb R^d).
\end{equation}
Equivalently, the equation and trace are expressed by
\begin{equation}\label{spacetimeweak}
-\int_0^\infty\langle W(t),\partial_t\eta(\cdot,t)\rangle\,dt
-\int_0^\infty\langle J(t),\nabla_x\eta(\cdot,t)\rangle\,dt
=\int_0^\infty\langle f(t),\eta(\cdot,t)\rangle\,dt+\langle g,\eta(\cdot,0)\rangle
\end{equation}
for every $\eta\in C_c^\infty(\mathbb R^d\times[0,\infty))$, together with the specified weak-$*$ trace. In \eqref{spacetimeweak}, the vector pairing is the sum of the componentwise pairings. Moreover, problem \eqref{combinationeq} admits a measure-valued mild solution given by
\begin{equation}\label{duhamel}
u(\cdot,t)=G_t*g+\int_0^tG_{t-\tau}*f(\cdot,\tau)\,d\tau,\quad t>0,
\end{equation}
where the time integral is understood in the weak-$*$ sense in $\mathcal M(\mathbb R^d)$. This solution is weak-$*$ continuous on $(0,\infty)$ and is unique up to equality for almost every $t>0$ in the class \eqref{solutionclass} satisfying the equation and trace above. If another solution is weak-$*$ continuous on $(0,\infty)$, equality holds for every $t>0$. For finite signed Borel measures $\mu,\nu\in\mathcal M(\mathbb R^d)$, their convolution is defined by
\begin{equation}\label{measureconvolution}
\langle\mu*\nu,\varphi\rangle=\int_{\mathbb R^d}\int_{\mathbb R^d}\varphi(x+y)\,\mu(dx)\nu(dy),\quad\varphi\in C_0(\mathbb R^d).
\end{equation}
Here, for every $t>0$, $G_t$ is a finite positive Borel measure on $\mathbb R^d$, and the family $(G_t)_{t>0}$ is the weak-$*$ continuous kernel constructed in Lemma~\ref{kernelcanonical} and characterized by \eqref{kerneltransform}, where
\begin{equation}
\widehat{G_t}(\xi):=\int_{\mathbb R^d}e^{i\xi\cdot x}\,G_t(dx).
\end{equation}
Moreover,
\begin{equation}\label{kernelmass}
\|G_t\|_{\mathcal M(\mathbb R^d)}=G_t(\mathbb R^d)=\frac{t^{\alpha-1}}{\Gamma(\alpha)},\quad t>0.
\end{equation}
In general, $G_t$ admits a subordinator representation. Let $(\mu_r)_{r>0}$ be the weakly continuous convolution semigroup of probability measures on $\mathbb R_+\times\mathbb R^d$ characterized by its joint Fourier-Laplace transform
\begin{equation}
\int_0^\infty\int_{\mathbb R^d}e^{-st+i\xi\cdot x}\,\mu_r(dt,dx)=\exp\left(-r\int_{\mathbb S^{d-1}}(s-ib\cdot\xi)^\alpha\Lambda(db)\right).
\end{equation}
By the disintegration theorem,
\begin{equation}
\mu_r(dt,dx)=h_\alpha(t,r)\nu_{r,t}(dx)\,dt,
\end{equation}
where $\nu_{r,t}$ is a regular conditional probability measure on $\mathbb R^d$ supported in $tC_\Lambda\subseteq\overline{B_t(0)}$ for almost every $(r,t)$, and $h_\alpha(\cdot,r)$ is the density of the $\alpha$-stable subordinator, characterized by
\begin{equation}\label{stabledensity}
\int_0^\infty e^{-\lambda y}h_\alpha(y,r)\,dy=e^{-r\lambda^\alpha},\quad\Re\lambda>0.
\end{equation}
The kernel $G_t$ is then given by
\begin{equation}\label{stablekernel_multi}
G_t(dx)=\int_0^\infty h_\alpha(t,r)\nu_{r,t}(dx)\,dr
\quad \text{for almost every }t>0.
\end{equation}
A jointly measurable disintegration is obtained by disintegrating the sigma-finite measure $dr\,\mu_r(dt,dx)$ with respect to $dr\,h_\alpha(t,r)\,dt$. For the endpoint case where $\Lambda=\delta_{b_0}$ for some direction $b_0\in\mathbb S^{d-1}$, we have
\begin{equation}\label{endpointkernels}
G_t=\frac{t^{\alpha-1}}{\Gamma(\alpha)}\delta_{tb_0}.
\end{equation}

\begin{proof}
Let
\begin{equation}
\zeta(\xi,s):=\int_{\mathbb S^{d-1}}(s-ib\cdot\xi)^\alpha\Lambda(db),
\end{equation}
where the principal branch of the complex power is used. Lemma~\ref{kernelcanonical} proves that $\zeta(\xi,s)\ne0$ for $\Re s>0$ and constructs the positive kernel with transform \eqref{kerneltransform}, mass \eqref{kernelmass}, and support in $tC_\Lambda$. The disintegration formula follows by applying Tonelli's theorem to $Q=\int_0^\infty\mu_r\,dr$. Next, for the endpoint case $\Lambda=\delta_{b_0}$, we define
\begin{equation}
G_t=\frac{t^{\alpha-1}}{\Gamma(\alpha)}\delta_{tb_0}.
\end{equation}
Then
\begin{equation}
\int_0^\infty e^{-st}\widehat{G_t}(\xi)\,dt
=\frac1{\Gamma(\alpha)}\int_0^\infty e^{-(s-ib_0\cdot\xi)t}t^{\alpha-1}\,dt=(s-ib_0\cdot\xi)^{-\alpha}=\frac1{\zeta(\xi,s)}.
\end{equation}
Uniqueness of the Fourier-Laplace transform and weak continuity allow us to identify this kernel with the representative in Lemma~\ref{kernelcanonical}.
For every $\mu\in\mathcal M(\mathbb R^d)$, convolution with $G_t$ satisfies
\begin{equation}\label{convolutionestimate}
\|G_t*\mu\|_{\mathcal M(\mathbb R^d)}\leq\|G_t\|_{\mathcal M(\mathbb R^d)}\|\mu\|_{\mathcal M(\mathbb R^d)}=\frac{t^{\alpha-1}}{\Gamma(\alpha)}\|\mu\|_{\mathcal M(\mathbb R^d)}.
\end{equation}
Consequently, the right-hand side of \eqref{duhamel} is well-defined as a weak-$*$ integral in $\mathcal M(\mathbb R^d)$. More precisely, for every $\varphi\in C_0(\mathbb R^d)$,
\begin{equation}\label{weakduhamel}
\left\langle\int_0^tG_{t-\tau}*f(\cdot,\tau)\,d\tau,\varphi\right\rangle
=\int_0^t\langle G_{t-\tau}*f(\cdot,\tau),\varphi\rangle\,d\tau.
\end{equation}
For completeness, for fixed $t$ and $\varphi\in C_0$, the function
$H_r(y)=\int\varphi(y+x)G_r(dx)$ is continuous from $r>0$ into $C_0$ in the uniform norm, by \eqref{canonicalkernel} and uniform continuity of $\varphi$. Consequently, $\tau\mapsto\langle f(\tau),H_{t-\tau}\rangle$ is measurable. Indeed, by \eqref{convolutionestimate} and
\eqref{uniformexpbound}, we have
\begin{equation}
\begin{split}
\int_0^t
\|G_{t-\tau}*f(\cdot,\tau)\|_{\mathcal M(\mathbb R^d)}
\,d\tau
&\leq
\frac1{\Gamma(\alpha)}
\int_0^t
(t-\tau)^{\alpha-1}
\|f(\cdot,\tau)\|_{\mathcal M(\mathbb R^d)}
\,d\tau\\
&\leq
\frac{M}{\Gamma(\alpha)}
\int_0^t
(t-\tau)^{\alpha-1}\tau^{-\beta}e^{a\tau}
\,d\tau
<\infty,
\end{split}
\end{equation}
since $0<\alpha<1$ and $0\leq\beta<1$.
Together with the measurability established above, this bound
and the Riesz representation theorem define the weak-$*$ integral
as a finite measure. Consequently, the right-hand side of
\eqref{duhamel} is well-defined and satisfies
\begin{equation}\label{solutionestimate}
\|u(\cdot,t)\|_{\mathcal M(\mathbb R^d)}
\leq
\frac{t^{\alpha-1}}{\Gamma(\alpha)}
\|g\|_{\mathcal M(\mathbb R^d)}+
\frac1{\Gamma(\alpha)}
\int_0^t
(t-\tau)^{\alpha-1}
\|f(\cdot,\tau)\|_{\mathcal M(\mathbb R^d)}
\,d\tau.
\end{equation}
For $\sigma>\max\{0,a\}$, Tonelli's theorem gives
\begin{equation}\label{weightedsolutionestimate}
\int_0^\infty e^{-\sigma t}\|u(t)\|_{\mathcal M}\,dt
\leq\sigma^{-\alpha}\left(\|g\|_{\mathcal M}+M\Gamma(1-\beta)(\sigma-a)^{\beta-1}\right).
\end{equation}
Thus the solution belongs to \eqref{solutionclass} with a singularity of order $t^{\alpha-1}$ at zero allowed. Weak-$*$ continuity of $G_t*g$ follows from \eqref{canonicalkernel}. To prove continuity of the source term at $t_0>0$, we split its integral into $\tau$ near zero, $t-\tau$ near zero, and the remaining interval. For $t$ near $t_0$, the first contribution is bounded by a constant times $\int_0^\delta\tau^{-\beta}\,d\tau$, and the second by a constant times $\int_0^{2\delta}r^{\alpha-1}\,dr$. Both tend to zero as $\delta$ tends to zero. On the remaining interval, we use continuity of $H_{t-\tau}$ and dominated convergence theorem. This proves weak-$*$ continuity on $(0,\infty)$. Next, we verify the Riemann-Liouville initial condition. Let
\begin{equation}
u(\cdot,t)=u_g(\cdot,t)+u_f(\cdot,t),
\end{equation}
where
\begin{equation}
u_g(\cdot,t)=G_t*g
\end{equation}
and
\begin{equation}
u_f(\cdot,t)=\int_0^tG_{t-\tau}*f(\cdot,\tau)\,d\tau.
\end{equation}
For $t>0$, we define a positive Borel measure $K_t$ by
\begin{equation}\label{Ktdefinition}
\langle K_t,\varphi\rangle:=\frac1{\Gamma(1-\alpha)}\int_{\mathbb S^{d-1}}\int_0^t(t-\tau)^{-\alpha}\int_{\mathbb R^d}\varphi(x+b(t-\tau))\,G_\tau(dx)\,d\tau\,\Lambda(db),
\end{equation}
for $\varphi\in C_0(\mathbb R^d)$. By \eqref{kernelmass}, we have
\begin{equation}
K_t(\mathbb R^d)=\frac1{\Gamma(1-\alpha)\Gamma(\alpha)}\int_0^t(t-\tau)^{-\alpha}\tau^{\alpha-1}\,d\tau=\frac{B(\alpha,1-\alpha)}{\Gamma(\alpha)\Gamma(1-\alpha)}=1.
\end{equation}
Furthermore, since $G_\tau$ is supported in $\overline{B_\tau(0)}$, the measure $K_t$ is supported in $\overline{B_t(0)}$. Hence $(K_t)_{t>0}$ is an approximate identity. Using \eqref{fractionalmaterialintegral}, we obtain
\begin{equation}
\int_{\mathbb S^{d-1}}I_b^{1-\alpha}u_g(\cdot,t)\,\Lambda(db)=K_t*g.
\end{equation}
Therefore, for every $\varphi\in C_0(\mathbb R^d)$, we have
\begin{equation}
|\langle K_t*g-g,\varphi\rangle|\leq\|g\|_{\mathcal M(\mathbb R^d)}\sup_{x\in\mathbb R^d,\ |y|\leq t}|\varphi(x+y)-\varphi(x)|.
\end{equation}
Since $\varphi \in C_0(\mathbb R^d)$, the right-hand side converges to zero as $t\to0^+$. Hence
\begin{equation}\label{homogeneousinitialtrace}
\lim_{t\to0^+}\int_{\mathbb S^{d-1}}I_b^{1-\alpha}u_g(\cdot,t)\,\Lambda(db)=g
\end{equation}
in the weak-$*$ topology of $\mathcal M(\mathbb R^d)$. It remains to show that the source term has zero fractional initial trace. From \eqref{uniformexpbound} and \eqref{convolutionestimate}, we know that for $0<t\leq1$, there exists a constant $C>0$ such that
\begin{equation}
\|u_f(\cdot,t)\|_{\mathcal M(\mathbb R^d)}
\leq\frac M{\Gamma(\alpha)}\int_0^t(t-\tau)^{\alpha-1}\tau^{-\beta}e^{a\tau}\,d\tau\leq Ct^{\alpha-\beta}.
\end{equation}
Consequently,
\begin{equation}
\begin{split}
\left\|\int_{\mathbb S^{d-1}}I_b^{1-\alpha}u_f(\cdot,t)\,\Lambda(db)\right\|_{\mathcal M(\mathbb R^d)}
&\leq\frac1{\Gamma(1-\alpha)}\int_0^t(t-\tau)^{-\alpha}\|u_f(\cdot,\tau)\|_{\mathcal M(\mathbb R^d)}\,d\tau\\
&\leq\frac{C B(1-\alpha,1+\alpha-\beta)}{\Gamma(1-\alpha)}t^{1-\beta}.
\end{split}
\end{equation}
Since $\beta<1$, the last expression converges to zero as $t\to0^+$. Thus,
\begin{equation}\label{forcinginitialtrace}
\lim_{t\to0^+}\int_{\mathbb S^{d-1}}I_b^{1-\alpha}u_f(\cdot,t)\,\Lambda(db)=0
\end{equation}
in $\mathcal M(\mathbb R^d)$. Combining \eqref{homogeneousinitialtrace} and \eqref{forcinginitialtrace}, we obtain
\begin{equation}
\lim_{t\to0^+}\int_{\mathbb S^{d-1}}I_b^{1-\alpha}u(\cdot,t)\,\Lambda(db)=g
\end{equation}
in the weak-$*$ topology of $\mathcal M(\mathbb R^d)$. Hence the Riemann-Liouville initial condition is satisfied.
We next verify that $u$ satisfies the equation. By \eqref{uniformexpbound}, for every $\sigma>a$, we have
\begin{equation}
\int_0^\infty e^{-\sigma t}\|f(\cdot,t)\|_{\mathcal M(\mathbb R^d)}\,dt
\leq M\Gamma(1-\beta)(\sigma-a)^{\beta-1}<\infty.
\end{equation}
Therefore, the measure-valued Laplace transform
\begin{equation}
\widetilde f(s):=\int_0^\infty e^{-st}f(\cdot,t)\,dt
\end{equation}
exists in the weak-$*$ sense for $\Re s>a$. The estimate \eqref{solutionestimate} also implies that $\widetilde u(s)$ exists for $\Re s>\max\{0,a\}$. Since every finite Borel measure on $\mathbb R^d$ defines a tempered distribution, the spatial Fourier transform may be understood in $\mathcal S'(\mathbb R^d)$. For a finite Borel measure $\mu$, we write
\begin{equation}
\widehat\mu(\xi):=\int_{\mathbb R^d}e^{i\xi\cdot x}\,\mu(dx).
\end{equation}
Taking the Laplace transform of \eqref{duhamel} and then the Fourier transform in space, and using \eqref{kerneltransform}, gives
\begin{equation}\label{transformedsolution}
\widehat{\widetilde u}(\xi,s)=\frac{\widehat{\widetilde f}(\xi,s)+\widehat g(\xi)}{\displaystyle\int_{\mathbb S^{d-1}}(s-ib\cdot\xi)^\alpha\Lambda(db)}.
\end{equation}
Equivalently,
\begin{equation}\label{transformedweakproblem}
\zeta(\xi,s)\widehat{\widetilde u}(\xi,s)=\widehat{\widetilde f}(\xi,s)+\widehat g(\xi).
\end{equation}
To identify this equation with the weak problem without requiring directional traces, we extend $u,f,W,J$ by zero to negative times. We denote the extensions by $u^0,f^0,W^0,J^0$. Since absolute exponential integrability holds, we can use Fubini's theorem, and obtain
\[
\mathcal{FL}\{I_b^{1-\alpha}u^0\}(\xi,s)
=(s-ib\cdot\xi)^{\alpha-1}\widehat{\widetilde u}(\xi,s).
\]
Consequently,
\[
\mathcal{FL}\{\partial_tW^0+\operatorname{div}_xJ^0\}(\xi,s)
=\zeta(\xi,s)\widehat{\widetilde u}(\xi,s).
\]
Equation \eqref{transformedweakproblem} and injectivity of the Fourier-Laplace transform of exponentially integrable causal distributions yield
\begin{equation}\label{causalproblem}
\partial_tW^0+\operatorname{div}_xJ^0=f^0+\delta_0(dt)\otimes g(dx).
\end{equation}
This is \eqref{spacetimeweak} and \eqref{weakcombinationeq} for $t>0$. Conversely, for a solution of \eqref{weakcombinationeq} with trace \eqref{weakRLinitialcondition}, integration by parts at $t=0$ gives us \eqref{causalproblem}. To justify this step, let us fix a smooth compactly supported spatial test function. Its pairing with $W$ has distributional time derivative $\langle f,\varphi\rangle+\langle J,\nabla\varphi\rangle\in L^1(0,T)$ and hence an absolutely continuous representative, whose initial value is fixed by the trace. This also justifies \eqref{spacetimeweak} for time-dependent test functions by approximation.
It remains to prove uniqueness. Let $u_1$ and $u_2$ be two measure-valued solutions in the class \eqref{solutionclass} satisfying \eqref{combinationeq} with the same data. Set
\begin{equation}
w=u_1-u_2.
\end{equation}
Then $w$ satisfies the homogeneous problem with zero initial trace. Taking its Fourier-Laplace transform gives
\begin{equation}
\zeta(\xi,s)\widehat{\widetilde w}(\xi,s)=0.
\end{equation}
Since $\zeta(\xi,s)\ne0$ for $\Re s>0$, we obtain
\begin{equation}
\widehat{\widetilde w}(\xi,s)=0.
\end{equation}
Since the Fourier transform is injective on finite Borel measures,
\begin{equation}
\widetilde w(s)=0.
\end{equation}
Pairing with an arbitrary $\varphi\in C_0(\mathbb R^d)$ gives
\begin{equation}
\int_0^\infty e^{-st}\langle w(\cdot,t),\varphi\rangle\,dt=0.
\end{equation}
Uniqueness of the scalar Laplace transform therefore yields
\begin{equation}
\langle w(\cdot,t),\varphi\rangle=0
\end{equation}
for almost every $t>0$. Next, we choose a countable dense subset of $C_0(\mathbb R^d)$ and discard the union of its exceptional null sets. By continuity of measure pairings, $w(\cdot,t)=0$ as a measure outside this single null set.
Thus the measure-valued mild solution is unique, up to equality for almost every $t>0$. If the solutions are weak-$*$ continuous on $(0,\infty)$, then the uniqueness holds for every $t>0$.
\end{proof}
\end{theorem}
Next, we prove continuous dependence on the data
\begin{proposition}[Continuous dependence on the data]\label{datastability}
Let $u_1,u_2$ be the solutions from Theorem~\ref{exuniq} for the same $\Lambda$ and data $(g_1,f_1)$ and $(g_2,f_2)$. Then
\begin{equation}\label{datastabilityestimate}
\|u_1(t)-u_2(t)\|_{\mathcal M}
\leq\frac{t^{\alpha-1}}{\Gamma(\alpha)}\|g_1-g_2\|_{\mathcal M}
+\frac1{\Gamma(\alpha)}\int_0^t(t-\tau)^{\alpha-1}\|f_1(\tau)-f_2(\tau)\|_{\mathcal M}\,d\tau.
\end{equation}
Moreover, for $\sigma>0$ such that the right-hand side below is finite,
\begin{equation}
\int_0^\infty e^{-\sigma t}\|u_1(t)-u_2(t)\|_{\mathcal M}\,dt
\leq\sigma^{-\alpha}\left(\|g_1-g_2\|_{\mathcal M}
+\int_0^\infty e^{-\sigma t}\|f_1(t)-f_2(t)\|_{\mathcal M}\,dt\right).
\end{equation}
\begin{proof}
It is enough to subtract the two Duhamel formulas and apply \eqref{convolutionestimate}. The weighted estimate follows from Tonelli's theorem and the identity $\int_0^\infty e^{-\sigma r}r^{\alpha-1}\,dr=\Gamma(\alpha)\sigma^{-\alpha}$.
\end{proof}
\end{proposition}
Since we are interested in connections between the fractional material derivative and stochastic processes, in the next proposition we find a necessary and sufficient condition on the source $f$ for the solution to have unit mass for all times, and separate sufficient conditions for positivity and weak convergence to $\delta_0$.
\begin{proposition}\label{probconsprop}
Consider problem \eqref{combinationeq} with the homogeneous Riemann-Liouville initial datum $g=0$. Let $G_t$ be the positive convolution kernel introduced in Theorem~\ref{exuniq}. Let
\begin{equation}
f:(0,\infty)\to\mathcal M(\mathbb R^d),\quad t\mapsto f(\cdot,t),
\end{equation}
be weak-$*$ measurable, and define
\begin{equation}\label{duhamel2}
u(\cdot,t)=\int_0^tG_{t-\tau}*f(\cdot,\tau)\,d\tau,\quad t>0,
\end{equation}
where the integral is understood in the weak-$*$ sense in $\mathcal M(\mathbb R^d)$. Assume that
\begin{equation}\label{weightedfintegrability}
\int_0^t(t-\tau)^{\alpha-1}\|f(\cdot,\tau)\|_{\mathcal M(\mathbb R^d)}\,d\tau<\infty,\quad t>0.
\end{equation}
Set
\begin{equation}
m(t):=f(\cdot,t)(\mathbb R^d).
\end{equation}
Then
\begin{equation}\label{massidentity}
u(\cdot,t)(\mathbb R^d)=I_t^\alpha m(t)=\frac1{\Gamma(\alpha)}\int_0^t(t-\tau)^{\alpha-1}m(\tau)\,d\tau.
\end{equation}
Consequently,
\begin{equation}\label{densitycond}
u(\cdot,t)(\mathbb R^d)=1,\quad t>0,
\end{equation}
if and only if
\begin{equation}\label{pdfcond}
f(\cdot,t)(\mathbb R^d)=\frac{t^{-\alpha}}{\Gamma(1-\alpha)}
\end{equation}
for almost every $t>0$. If, in addition,
\begin{equation}\label{sourcepositivity}
f(\cdot,t)\text{ is a positive Borel measure for almost every }t>0,
\end{equation}
then $u(\cdot,t)$ is a positive Borel measure for every $t>0$. Therefore, under \eqref{pdfcond} and \eqref{sourcepositivity}, $u(\cdot,t)$ is a probability measure on $\mathbb R^d$ for every $t>0$. Furthermore, under assumptions \eqref{pdfcond} and \eqref{sourcepositivity}, define the normalized source measures by
\begin{equation}\label{normalizedsource}
\nu_t(dx):=\Gamma(1-\alpha)t^\alpha f(dx,t).
\end{equation}
After modifying $f(\cdot,t)$ on a set of times of Lebesgue measure zero, if necessary, the measures $\nu_t$ may be taken to be probability measures for every $t>0$. If
\begin{equation}\label{sourceconcentration}
\nu_t\rightharpoonup\delta_0\quad\text{as }t\to0^+,
\end{equation}
then
\begin{equation}\label{solutionconcentration}
u(\cdot,t)\rightharpoonup\delta_0\quad\text{as }t\to0^+.
\end{equation}
The weak convergence in \eqref{solutionconcentration} is an additional probabilistic condition and is distinct from the homogeneous Riemann-Liouville initial condition imposed in problem \eqref{combinationeq}.
\end{proposition}
\begin{proof}
By \eqref{weightedfintegrability} and the identity
\begin{equation}
\|G_t\|_{\mathcal M(\mathbb R^d)}=G_t(\mathbb R^d)=\frac{t^{\alpha-1}}{\Gamma(\alpha)},
\end{equation}
the weak-$*$ integral in \eqref{duhamel2} is well-defined as a finite signed Borel measure. Indeed, we have
\begin{equation}
\|u(\cdot,t)\|_{\mathcal M(\mathbb R^d)}
\leq\int_0^t\|G_{t-\tau}*f(\cdot,\tau)\|_{\mathcal M(\mathbb R^d)}\,d\tau \leq\frac1{\Gamma(\alpha)}\int_0^t(t-\tau)^{\alpha-1}\|f(\cdot,\tau)\|_{\mathcal M(\mathbb R^d)}\,d\tau<\infty.
\end{equation}
For finite Borel measures $\mu$ and $\nu$, the total mass of their convolution satisfies
\begin{equation}
(\mu*\nu)(\mathbb R^d)=\mu(\mathbb R^d)\nu(\mathbb R^d).
\end{equation}
Therefore, using \eqref{duhamel2}, \eqref{weightedfintegrability}, and Fubini's theorem, we obtain
\begin{equation}\label{massidentityproof}
\begin{split}
u(\cdot,t)(\mathbb R^d)
=\int_0^t(G_{t-\tau}*f(\cdot,\tau))(\mathbb R^d)\,d\tau&=\int_0^tG_{t-\tau}(\mathbb R^d)f(\cdot,\tau)(\mathbb R^d)\,d\tau\\&=\frac1{\Gamma(\alpha)}\int_0^t(t-\tau)^{\alpha-1}m(\tau)\,d\tau=I_t^\alpha m(t).
\end{split}
\end{equation}
Although $1\notin C_0(\mathbb R^d)$, evaluation of the total mass is justified by bounded compactly supported cutoffs converging pointwise to $1$, followed by dominated convergence using \eqref{weightedfintegrability}.
This proves \eqref{massidentity}. Assume first that \eqref{pdfcond} holds. Then
\begin{equation}
u(\cdot,t)(\mathbb R^d)
=\frac1{\Gamma(\alpha)\Gamma(1-\alpha)}\int_0^t(t-\tau)^{\alpha-1}\tau^{-\alpha}\,d\tau=\frac{B(\alpha,1-\alpha)}{\Gamma(\alpha)\Gamma(1-\alpha)}=1.
\end{equation}
Hence \eqref{densitycond} holds. Conversely, suppose that
\begin{equation}
u(\cdot,t)(\mathbb R^d)=1,\quad t>0.
\end{equation}
Then \eqref{massidentityproof} gives
\begin{equation}
I_t^\alpha m(t)=1.
\end{equation}
The function $m$ is locally integrable. Indeed, we have
\begin{equation}
|m(t)|\leq\|f(\cdot,t)\|_{\mathcal M(\mathbb R^d)},
\end{equation}
and the weighted integrability assumption \eqref{weightedfintegrability} implies local integrability in time. Applying the Riemann-Liouville derivative $D_t^\alpha$ and using
\begin{equation}
D_t^\alpha I_t^\alpha m(t)=m(t)
\end{equation}
for almost every $t>0$, together with
\begin{equation}
D_t^\alpha1=\frac{t^{-\alpha}}{\Gamma(1-\alpha)},
\end{equation}
we obtain
\begin{equation}
m(t)=\frac{t^{-\alpha}}{\Gamma(1-\alpha)}
\end{equation}
for almost every $t>0$. This proves \eqref{pdfcond}. Suppose now that \eqref{sourcepositivity} holds. Since $G_t$ is a positive Borel measure, for almost every $\tau>0$ the convolution
\begin{equation}
G_{t-\tau}*f(\cdot,\tau)
\end{equation}
is a positive Borel measure. Therefore its weak-$*$ time integral is also a positive Borel measure, and hence
\begin{equation}
u(\cdot,t)\geq0
\end{equation}
in the sense of measures. More explicitly, for every nonnegative $\varphi\in C_0(\mathbb R^d)$, we have
\begin{equation}
\langle u(\cdot,t),\varphi\rangle=\int_0^t\langle G_{t-\tau}*f(\cdot,\tau),\varphi\rangle\,d\tau\geq0.
\end{equation}
Together with \eqref{densitycond}, this shows that $u(\cdot,t)$ is a probability measure for every $t>0$. It remains to prove the weak-convergence statement. Let
\begin{equation}\label{normalizedkernel}
P_r:=\Gamma(\alpha)r^{1-\alpha}G_r,\quad r>0.
\end{equation}
We use $P_r$ to distinguish this normalized kernel from $K_t$ in \eqref{Ktdefinition}. By the mass identity for $G_r$, we obtain
\begin{equation}
P_r(\mathbb R^d)=1,
\end{equation}
so $P_r$ is a probability measure. Moreover,
\begin{equation}
\operatorname{supp}P_r\subseteq\overline{B_r(0)}.
\end{equation}
Consequently,
\begin{equation}\label{kernelweakconvergence}
P_r\rightharpoonup\delta_0\quad\text{as }r\to0^+.
\end{equation}
Under \eqref{pdfcond} and \eqref{sourcepositivity},
\begin{equation}
\nu_t(dx)=\Gamma(1-\alpha)t^\alpha f(dx,t)
\end{equation}
is a probability measure. Equivalently,
\begin{equation}
f(dx,t)=\frac{t^{-\alpha}}{\Gamma(1-\alpha)}\nu_t(dx).
\end{equation}
Similarly, by \eqref{normalizedkernel},
\begin{equation}
G_r=\frac{r^{\alpha-1}}{\Gamma(\alpha)}P_r.
\end{equation}
Using these identities in \eqref{duhamel2} and making the change of variables $\tau=ts$, we obtain the following identity of probability measures:
\begin{equation}\label{mixturerepresentation}
u(\cdot,t)=\int_0^1\frac{(1-s)^{\alpha-1}s^{-\alpha}}{\Gamma(\alpha)\Gamma(1-\alpha)}\left(P_{t(1-s)}*\nu_{ts}\right)\,ds.
\end{equation}
The weight in \eqref{mixturerepresentation} integrates to one, since
\begin{equation}
\int_0^1\frac{(1-s)^{\alpha-1}s^{-\alpha}}{\Gamma(\alpha)\Gamma(1-\alpha)}\,ds=\frac{B(\alpha,1-\alpha)}{\Gamma(\alpha)\Gamma(1-\alpha)}=1.
\end{equation}
Let $\psi\in C_b(\mathbb R^d)$. For every fixed $s\in(0,1)$, as $t\to0^+$,
\begin{equation}
P_{t(1-s)}\rightharpoonup\delta_0
\end{equation}
by \eqref{kernelweakconvergence}, while
\begin{equation}
\nu_{ts}\rightharpoonup\delta_0
\end{equation}
by \eqref{sourceconcentration}. Since convolution is continuous with respect to weak convergence of probability measures,
\begin{equation}
P_{t(1-s)}*\nu_{ts}\rightharpoonup\delta_0.
\end{equation}
Hence
\begin{equation}
\int_{\mathbb R^d}\psi(x)\left(P_{t(1-s)}*\nu_{ts}\right)(dx)\longrightarrow\psi(0)\quad\text{as }t\to0^+.
\end{equation}
Moreover,
\begin{equation}
\left|\int_{\mathbb R^d}\psi(x)\left(P_{t(1-s)}*\nu_{ts}\right)(dx)\right|\leq\|\psi\|_\infty,
\end{equation}
because $P_{t(1-s)}*\nu_{ts}$ is a probability measure. Therefore, by the dominated convergence theorem applied to \eqref{mixturerepresentation}, we obtain
\begin{equation}
\lim_{t\to0^+}\langle u(\cdot,t),\psi\rangle=\lim_{t\to0^+}\int_{\mathbb R^d}\psi(x)\,u(dx,t)=\psi(0).
\end{equation}
Thus,
\begin{equation}
u(\cdot,t)\rightharpoonup\delta_0\quad\text{as }t\to0^+.
\end{equation}
This weak convergence is a statement about the probabilistic concentration of the probability measures $u(\cdot,t)$ and should not be confused with the homogeneous Riemann-Liouville initial condition
\begin{equation}
\lim_{t\to0^+}\int_{\mathbb S^{d-1}}I_b^{1-\alpha}u(\cdot,t)\,\Lambda(db)=0,
\end{equation}
which is the initial condition associated with \eqref{combinationeq} when $g=0$. Under the positivity and mass assumptions, this trace also follows directly from $\|u(\cdot,t)\|_{\mathcal M}=1$ and the bound $t^{1-\alpha}/\Gamma(2-\alpha)$ for its fractional material integral. This completes the proof.
\end{proof}
Finally, we prove stbility with respect to the directional distribution.
\begin{theorem}[Stability with respect to the directional distribution]\label{angularstability}
Let $\Lambda_n\rightharpoonup\Lambda$ be probability measures on $\mathbb S^{d-1}$. Let $G_t^{(n)}$ and $G_t$ be the continuous representatives from Lemma~\ref{kernelcanonical}. Then
\begin{equation}\label{rhostability}
\rho_{\Lambda_n}\rightharpoonup\rho_\Lambda.
\end{equation}
Moreover, for every $\varphi\in C_b(\mathbb R^d)$ and $0<\varepsilon<T<\infty$,
\begin{equation}\label{kernelangularstability}
\sup_{\varepsilon\leq t\leq T}
|\langle G_t^{(n)}-G_t,\varphi\rangle|\longrightarrow0,
\end{equation}
and for fixed data $g,f$ satisfying assumptions of Theorem~\ref{exuniq}, then the corresponding solutions satisfy
\begin{equation}\label{solutionangularstability}
u_n(t)\stackrel{*}{\rightharpoonup}u(t)
\quad\text{in }\mathcal M(\mathbb R^d),\quad t>0.
\end{equation}
If $g=0$ and $f$ is positive with mass $t^{-\alpha}/\Gamma(1-\alpha)$ almost everywhere, this convergence is equivalent to the weak convergence of probability measures.
\begin{proof}
Let $Q_n,Q$ be the same potential measures as in the proof of Lemma~\ref{kernelcanonical}. Then, the measures
\[
P_n(dt,dx)=e^{-t}Q_n(dt,dx),\quad P(dt,dx)=e^{-t}Q(dt,dx)
\]
are probability measures on $\mathbb R^{d+1}$. Their characteristic functions are
\[
\widehat P_n(\eta,\xi)=\frac1{\zeta_{\Lambda_n}(\xi,1-i\eta)},\quad
\widehat P(\eta,\xi)=\frac1{\zeta_\Lambda(\xi,1-i\eta)}.
\]
For fixed $(\eta,\xi)$ the integrand defining $\zeta$ is continuous in $b$ on the compact sphere. Thus $\widehat P_n\to\widehat P$ pointwise, and L\'evy's continuity theorem yields $P_n\rightharpoonup P$. All these measures have full mass for $t>0$. The map $(t,x)\mapsto x/t$, defined arbitrarily when $t=0$, is continuous outside a $P$-null set. Their pushforwards are $\rho_{\Lambda_n}$ and $\rho_\Lambda$, respectively, since
\[
\int_0^\infty e^{-t}\frac{t^{\alpha-1}}{\Gamma(\alpha)}\,dt=1.
\]
Thus, the continuous mapping theorem therefore proves \eqref{rhostability}. All the measures $\rho_{\Lambda_n}$ and $\rho_\Lambda$ are supported in $\overline{B_1(0)}$. The family of functions $z\mapsto\varphi(tz)$, $t\in[\varepsilon,T]$, is compact in $C(\overline{B_1(0)})$ and the map from $t$ into this space is continuous by uniform continuity of $\varphi$ on $\overline{B_T(0)}$. To see that the convergence is uniform on this family, fix
$\delta>0$. By compactness, there exist
$h_1,\ldots,h_m\in\mathcal K$ such that every
$h\in\mathcal K$ satisfies $\|h-h_j\|_\infty<\delta$
for some $j$. Since $\rho_{\Lambda_n}$ and $\rho_\Lambda$
are probability measures,
\[
\begin{aligned}
\left|\int h\,d\rho_{\Lambda_n}-\int h\,d\rho_\Lambda\right|
&\leq
\left|\int h_j\,d\rho_{\Lambda_n}
      -\int h_j\,d\rho_\Lambda\right|
+2\delta.
\end{aligned}
\]
Consequently,
\[
\sup_{h\in\mathcal K}
\left|\int h\,d\rho_{\Lambda_n}-\int h\,d\rho_\Lambda\right|
\leq
\max_{1\leq j\leq m}
\left|\int h_j\,d\rho_{\Lambda_n}
      -\int h_j\,d\rho_\Lambda\right|
+2\delta.
\]
The finite maximum tends to zero by weak convergence.
Taking the upper limit as $n\to\infty$ and then letting
$\delta\downarrow0$ proves theuniform convergence. Convergence in  \eqref{kernelangularstability} follows from \eqref{canonicalkernel}. Applying this first with $r=t$, $\mu=g$, and then with
$r=t-\tau$, $\mu=f(\tau)$ proves convergence of the
homogeneous term and pointwise convergence of the source
integrand for $\tau\in(0,t)$. Moreover, by
\eqref{convolutionestimate},
\[
|\langle(G_{t-\tau}^{(n)}-G_{t-\tau})*f(\tau),\varphi\rangle|
\leq
\frac{2\|\varphi\|_\infty}{\Gamma(\alpha)}
(t-\tau)^{\alpha-1}\|f(\tau)\|_{\mathcal M}.
\]
The right-hand side is integrable on $(0,t)$ by
\eqref{uniformexpbound}. Using dominated convergence in the
Duhamel formula therefore we obtain
$\langle u_n(t)-u(t),\varphi\rangle\to0$.
Since $\varphi\in C_0(\mathbb R^d)$ was arbitrary,
this proves \eqref{solutionangularstability}. Under the last assumptions the solutions are probability measures by Proposition~\ref{probconsprop}. Vague convergence to a probability measure then implies weak convergence of probability measures.
\end{proof}
\end{theorem}

\section{Examples}
In all examples below, we understood limits of spatial measures at $t=0$ weakly. The probability solutions have homogeneous averaged Riemann--Liouville trace. Indeed, we have
\[
\left\|\int_{\mathbb S^{d-1}}I_b^{1-\alpha}u(t)\,\Lambda(db)\right\|_{\mathcal M}
\leq\frac{t^{1-\alpha}}{\Gamma(2-\alpha)}\longrightarrow0.
\]
Thus their probabilistic initial condition is consistent with Theorem~\ref{exuniq} with $g=0$. It is interesting that there exist several types of L\'evy walks that yield solutions $u=u(x,t)$ in an exact, analytical form. We start with one-dimensional examples with a fixed velocity $b=1$, equivalently $d=1$ and $\Lambda=\delta_1$.
In the next few examples, we will show the applications of the previous theorem to equations governing L\'evy walks.
\begin{example}
Consider the scaling limit of wait-first L\'evy walk. Then, its PDF $u(x,t)$ satisfies \cite{JurlewiczKernMeerschaertScheffler2012}
\begin{equation}
\begin{cases}
\left(\dfrac\partial{\partial t}+\dfrac\partial{\partial x}\right)^\alpha u(x,t)=\dfrac{t^{-\alpha}}{\Gamma(1-\alpha)}\delta(x),&x\in\mathbb R,\ t>0,\\
\displaystyle\lim_{t\to0^+}u(x,t)=\delta(x),
\end{cases}
\end{equation}
where $\delta(x)$ is the Dirac delta. In this case, we have
\begin{equation}
u(x,t)=\frac{\sin(\pi\alpha)}\pi x^{\alpha-1}(t-x)^{-\alpha}H(x)H(t-x).
\end{equation}
\end{example}
\begin{example}
Consider the scaling limit of jump-first L\'evy walk. Then, its PDF $u(x,t)$ satisfies \cite{JurlewiczKernMeerschaertScheffler2012}
\begin{equation}
\begin{cases}
\left(\dfrac\partial{\partial t}+\dfrac\partial{\partial x}\right)^\alpha u(x,t)
=\dfrac\alpha{\Gamma(1-\alpha)}\displaystyle\int_t^\infty\delta(x-r)r^{-\alpha-1}\,dr\\
\hfill=\dfrac\alpha{\Gamma(1-\alpha)}x^{-\alpha-1}H(x-t),\quad x\in\mathbb R,\ t>0,\\
\displaystyle\lim_{t\to0^+}u(x,t)=\delta(x).
\end{cases}
\end{equation}
In this case we have
\begin{equation}
u(x,t)=\frac{\sin(\pi\alpha)}\pi\frac1x\left(\frac t{x-t}\right)^\alpha H(x-t).
\end{equation}
\end{example}
\begin{example}
Consider the scaling limit of standard L\'evy walk. Then $u(x,t)$ satisfies \cite{JurlewiczKernMeerschaertScheffler2012}
\begin{equation}
\begin{cases}
\left(\dfrac\partial{\partial t}+\dfrac\partial{\partial x}\right)^\alpha u(x,t)=\dfrac{t^{-\alpha}}{\Gamma(1-\alpha)}\delta(x-t),&x\in\mathbb R,\ t>0,\\
\displaystyle\lim_{t\to0^+}u(x,t)=\delta(x).
\end{cases}
\end{equation}
In this case, we have the solution
\begin{equation}
u(x,t)=\delta(x-t).
\end{equation}
\end{example}
Other explicit formulas were derived in \cite{MagdziarzZorawik2017} using probabilistic methods. In all of these cases, the measure $\Lambda$ is the uniform measure on the sphere $\mathbb S^{d-1}$ with $d\geq2$ with marginal distribution given by
\begin{equation}
\Lambda_1(db_1)=c_d(1-b_1^2)^{(d-3)/2}\,db_1,\quad-1<b_1<1,
\end{equation}
where $c_d=\frac1{\sqrt\pi}\frac{\Gamma(d/2)}{\Gamma((d-1)/2)}$. In the following formulas $n\in\mathbb N_0$, and $H$ denotes the Heaviside function. For positive $x$, all hypergeometric expressions below at $1/x^2$ are interpreted as the boundary values obtained from $1/(x+i\varepsilon)^2$ as $\varepsilon\downarrow0$. This fixes the side of the branch cut.

\begin{example}\label{wfex}
For the \emph{wait first} L\'evy walk, we have \cite{MagdziarzZorawik2017}
\begin{equation}
\begin{cases}
\left(\dfrac\partial{\partial t}+b\cdot\nabla\right)_\Lambda^\alpha u(x,t)=\dfrac{t^{-\alpha}}{\Gamma(1-\alpha)}\delta^{(d)}(x),\\
\displaystyle\lim_{t\to0^+}u(x,t)=\delta^{(d)}(x).
\end{cases}
\end{equation}
If $d=2n+3$, then the solution has the explicit self-similar form
\begin{equation}
u(x,t)=\frac{\Gamma(n+3/2)}{2\pi^{n+3/2}t|x|^{2n+2}}\phi\left(\frac{|x|}t\right),
\end{equation}
where
\begin{equation}
\phi(\sqrt y)=\begin{cases}
\dfrac{2\sqrt\pi}{\Gamma(n+3/2)}y^{n+1}(-1)^{n+1}\dfrac{d^{n+1}}{dy^{n+1}}\phi_1(\sqrt y),&y\in(0,1),\\
0,&\text{otherwise},
\end{cases}
\end{equation}
where
\begin{equation}
\phi_1(x)=-\frac1{\pi|x|}\Im\left(\frac1{{}_2F_1(-\alpha/2,(1-\alpha)/2;3/2+n;1/x^2)}\right),
\end{equation}
where ${}_2F_1$ is the Gaussian hypergeometric function.
If $d=2n+2$, then the solution has the explicit self-similar form
\begin{equation}
u(x,t)=\frac{\Gamma(n+1)}{2\pi^{n+1}t|x|^{2n+1}}\phi\left(\frac{|x|}t\right),
\end{equation}
where
\begin{equation}
\phi(\sqrt y)=\begin{cases}
\dfrac{2\sqrt\pi}{\Gamma(n+1)}y^{n+1/2}D_-^{n+1/2}\phi_1(\sqrt y),&y\in(0,1),\\
0,&\text{otherwise},
\end{cases}
\end{equation}
where
\begin{equation}
D_-^{n+1/2}f(x)=\left(-\frac d{dx}\right)^{n+1}\frac1{\pi^{1/2}}\int_x^\infty\frac{f(y)}{(y-x)^{1/2}}\,dy
\end{equation}
is the right-sided Riemann-Liouville derivative of order $n+1/2$ and
\begin{equation}
\phi_1(x)=-\frac1{\pi|x|}\Im\left(\frac1{{}_2F_1(-\alpha/2,(1-\alpha)/2;1+n;1/x^2)}\right).
\end{equation}
\end{example}

\begin{example}
In the case of the \emph{jump first} L\'evy walk, the main equation has the form \cite{MagdziarzZorawik2017}
\begin{equation}
\begin{cases}
\left(\dfrac\partial{\partial t}+b\cdot\nabla\right)_\Lambda^\alpha u(x,t)
=\dfrac\alpha{\Gamma(1-\alpha)}\displaystyle\int_t^\infty\left(\int_{\mathbb S^{d-1}}\delta^{(d)}(x-br)\,\Lambda(db)\right)r^{-\alpha-1}\,dr,\\
\displaystyle\lim_{t\to0^+}u(x,t)=\delta^{(d)}(x).
\end{cases}
\end{equation}
If $d=2n+3$, then the solution has the explicit self-similar form
\begin{equation}
u(x,t)=\frac{\Gamma(n+3/2)}{2\pi^{n+3/2}t|x|^{2n+2}}\phi\left(\frac{|x|}t\right),
\end{equation}
where
\begin{equation}
\phi(\sqrt y)=\begin{cases}
\dfrac{2\sqrt\pi}{\Gamma(n+3/2)}y^{n+1}(-1)^{n+1}\dfrac{d^{n+1}}{dy^{n+1}}\phi_1(\sqrt y),&y\in(0,\infty),\\
0,&\text{otherwise},
\end{cases}
\end{equation}
where
\begin{equation}
\phi_1(x)=\frac c{\pi|x|^{\alpha+1}}\Im\left(\frac{\cos(\pi\alpha/2)+i\sin(\pi\alpha/2)}{{}_2F_1(-\alpha/2,(1-\alpha)/2;3/2+n;1/x^2)}\right),
\end{equation}
where $c=\frac{\cos(\alpha\pi/2)\Gamma(2-\alpha)\Gamma((1+\alpha)/2)\Gamma(3/2+n)}{(1-\alpha)\sqrt\pi\Gamma(1-\alpha)\Gamma(3/2+\alpha/2+n)}$.
If $d=2n+2$, then the solution has the explicit self-similar form
\begin{equation}
u(x,t)=\frac{\Gamma(n+1)}{2\pi^{n+1}t|x|^{2n+1}}\phi\left(\frac{|x|}t\right),
\end{equation}
where
\begin{equation}
\phi(\sqrt y)=\begin{cases}
\dfrac{2\sqrt\pi}{\Gamma(n+1)}y^{n+1/2}D_-^{n+1/2}\phi_1(\sqrt y),&y\in(0,\infty),\\
0,&\text{otherwise},
\end{cases}
\end{equation}
where
\begin{equation}
\phi_1(x)=\frac c{\pi|x|^{\alpha+1}}\Im\left(\frac{\cos(\pi\alpha/2)+i\sin(\pi\alpha/2)}{{}_2F_1(-\alpha/2,(1-\alpha)/2;1+n;1/x^2)}\right),
\end{equation}
where $c=\frac{\cos(\alpha\pi/2)\Gamma(2-\alpha)\Gamma((1+\alpha)/2)\Gamma(1+n)}{(1-\alpha)\sqrt\pi\Gamma(1-\alpha)\Gamma(1+\alpha/2+n)}$.
\end{example}

\begin{example}
Now let us consider standard L\'evy walk. Then, the governing equation is of the form \cite{MagdziarzZorawik2017}
\begin{equation}
\begin{cases}
\left(\dfrac\partial{\partial t}+b\cdot\nabla\right)_\Lambda^\alpha u(x,t)=\dfrac{t^{-\alpha}}{\Gamma(1-\alpha)}\displaystyle\int_{\mathbb S^{d-1}}\delta^{(d)}(x-bt)\,\Lambda(db),\\
\displaystyle\lim_{t\to0^+}u(x,t)=\delta^{(d)}(x).
\end{cases}
\end{equation}
If $d=2n+3$, then the solution has the explicit self-similar form
\begin{equation}
u(x,t)=\frac{\Gamma(n+3/2)}{2\pi^{n+3/2}t|x|^{2n+2}}\phi\left(\frac{|x|}t\right),
\end{equation}
where
\begin{equation}
\phi(\sqrt y)=\begin{cases}
\dfrac{2\sqrt\pi}{\Gamma(n+3/2)}y^{n+1}(-1)^{n+1}\dfrac{d^{n+1}}{dy^{n+1}}\phi_1(\sqrt y),&y\in(0,1),\\
0,&\text{otherwise},
\end{cases}
\end{equation}
where
\begin{equation}
\phi_1(x)=-\frac1{\pi|x|}\Im\left(\frac{{}_2F_1((1-\alpha)/2,1-\alpha/2;3/2+n;1/x^2)}{{}_2F_1(-\alpha/2,(1-\alpha)/2;3/2+n;1/x^2)}\right).
\end{equation}
If $d=2n+2$, then the solution has the explicit self-similar form
\begin{equation}
u(x,t)=\frac{\Gamma(n+1)}{2\pi^{n+1}t|x|^{2n+1}}\phi\left(\frac{|x|}t\right),
\end{equation}
where
\begin{equation}
\phi(\sqrt y)=\begin{cases}
\dfrac{2\sqrt\pi}{\Gamma(n+1)}y^{n+1/2}D_-^{n+1/2}\phi_1(\sqrt y),&y\in(0,1),\\
0,&\text{otherwise},
\end{cases}
\end{equation}
where
\begin{equation}
\phi_1(x)=-\frac1{\pi|x|}\Im\left(\frac{{}_2F_1((1-\alpha)/2,1-\alpha/2;1+n;1/x^2)}{{}_2F_1(-\alpha/2,(1-\alpha)/2;1+n;1/x^2)}\right).
\end{equation}
\end{example}

\section*{Acknowledgments}
The author thanks Łukasz Płociniczak for all the help and discussions during the preparation of this paper and for careful reading of the first version of the manuscript.

This work has been supported by the National Science Centre, Poland (NCN) under the grant Sonata Bis with a number NCN 2020/38/E/ST1/00153.

The author used ChatGPT (OpenAI) for language editing and literature overview. The author is responsible for the final content of the manuscript, including the accuracy of the mathematical results and proofs and the precision of the references.
\bibliographystyle{plain}
\bibliography{bibliography.bib}
\end{document}